\documentclass[12pt]{article}
\usepackage{amssymb}
\usepackage{authblk}
\usepackage{seqsplit}
\usepackage{amsmath}
\usepackage{amsthm}  
\usepackage{derivative}
\usepackage{xspace}
\usepackage[dvipsnames]{xcolor}
\usepackage{cancel}
\usepackage{hyperref}
\usepackage[english]{babel}
\author{Ammar Hussain}
\affil{Department of Mathematics, Lahore University of Management Sciences, Lahore, Pakistan}
\affil{22100150@lums.edu.pk}
\title{Existence Conditions for Dilation Families}
\date{}
\begin{document}  

\maketitle

\newcommand{\al}{\alpha}
\newcommand{\be}{\beta} 
\newcommand{\ga}{\gamma}
\newcommand{\Ga}{\Gamma}
\newcommand{\de}{\delta}
\newcommand{\eps}{\epsilon}
\newcommand{\la}{\lambda}
\newcommand{\La}{\Lambda}
\newcommand{\om}{\omega}
\newcommand{\sig}{\sigma}
\newcommand{\thr}{\theta}
\newcommand{\fa}{\;\forall\;}
\newcommand{\ri}{\rightarrow}
\newcommand{\Ri}{\Rightarrow}
\newcommand{\charac}{\textrm{char}}
\newcommand{\ma}{\mathbb}
\newcommand{\ti}{\times}
\newcommand{\cd}{\cdot}
\newcommand{\nf}{\infty}
\newcommand{\nul}{\emptyset}
\newcommand{\tb}{\textbf}
\newcommand{\bs}{\boldsymbol}
\newcommand{\ayy}{\mathcal{A}}
\newcommand{\con}{\mathcal{C}}
\newcommand{\tr}{\textrm}
\newtheorem{prop}{Proposition}[section]
\newtheorem{ex}{Example}[section]
\newtheorem{defi}{Definition}[section]
\newtheorem{co}{Corollary}[prop]
\newtheorem{lem}[prop]{Lemma}
\newcommand{\uup}{\underline{\underline{Proof:}}\\\\}  

\begin{abstract}This article describes a structure that metric spaces can be equipped with so that they resemble normed vector spaces and examines necessary and sufficient conditions for the existence of such a structure on a general metric space.\end{abstract}

\section{Introduction}
Many ideas and structures in mathematics emerge as a way of filling the space 'between' concepts that have already been established and studied thoroughly. A classic example of this is the concept of uniform spaces, that came about as an attempt to generalize notions of 'uniformity' to a class of topological spaces beyond just metric spaces.\\

The concept of a metric space with a dilation family constitutes a similar sort of bridge between already established concepts. Specifically, it is a structure that is stronger than a metric space but weaker than a normed real vector space. Given a metric space (not necessarily equipped with any inner algebraic operation), a dilation family on the space is like an action by the nonnegative reals on the space that satisfies an analog of the homogeneity property of a vector norm (i.e. that $\bs{||}a \vec{v}\bs{||}=|a|\;\bs{||}\vec{v}\bs{||}$ always). Similar concepts have been explored already in papers by Marius Buliga, such as that of a 'dilatation structure' which is somewhere "between a group and a differential structure" \cite{dilatations}. One structure already defined by M. Buliga that is very similar to that of a metric space with a dilation family is a 'normed conical group' \cite{intrinsic}. A normed conical group is essentially a normed vector space where the vector addition operation is allowed to be noncommutative. In this sense, it is a stronger structure than a metric space with a dilation family. We will show that an exact condition for a metric group with a dilation family can be formulated to determine whether or not a norm exists that makes it a normed conical group.\\

As for determining whether a metric space can be equipped with a dilation family, we will demonstrate that any metrizable cone space can be equipped with a metric that preserves its topology and under which a dilation family exists about the cone's center. We will then go on to show that the existence of a special kind of null-homotopy on a locally compact metric space guarantees that the space is equivalent to a metrizable cone, meaning that it can be equipped with a dilation family (there are two cases here, depending on the size of the dilation family). Lastly, we will provide one result that gives a necessary and sufficient characterization for metric spaces to possess dilation families - this characterization relies on the existence of a special family of bi-Lipschitz maps on a metric space.\\ 

\tb{Acknowledgement:} I would like to thank Dr. Marius Buliga for pointing me to research relevant to this topic and my undergraduate Senior Project supervisor, Dr. Haniya Azam, for encouraging me to pursue this topic and giving me important advice during my first time working on an independent project.\\ 
 
\section{Dilation Families, Properties, and Characterization Results}

This section looks at some properties of dilation families and demonstrates some conditions on a metric space necessary and/ or sufficient for the existence of an associated dilation family.\\

\begin{defi} Let $(X,d_1)$ and $(Y,d_2)$ be a pair of metric spaces. Let $\al$ be a nonnegative real number. If $h:X\ri Y$ is a map such that $d_2(h(x_1),h(x_2))=\al d_1(x_1,x_2), \fa x_1,x_2\in X$, then $h$ is said to be a dilation from $(X,d_1)$ to $(Y,d_2)$ of scale $\al$.\\ \end{defi}

It is trivial to see that, just like isometries, all dilations are continuous and all dilations of positive scale are injective.\\

The central object of study here will be something I refer to as a 'dilation family'. Given a metric space $(X,d)$ and an index set $I\subset \ma R_{\geq 0}$ this is an indexed set of functions $\{T_{\al}\}_{\al\in I}\subset X^X$ that 'scale distances exactly' about a fixed point $x_0\in X$, and satisfy the property that $T_{\al}\circ T_{\be}=T_{\al\be}$ always.\\ 

More precisely, we have the following definition:\\

\begin{defi} Let $(X,d)$ be a metric space. Let $x_0$ be a point in $X$. Given any index set $I$ closed under multiplication and any collection $\ayy=\{T_{\al}\}_{\al\in I}\subset X^X$, we call $\ayy$ a dilation family in $X$ about $x_0$ if:\\ 

$1)\; T_{\al}$ is a dilation under $d$ of scale $\al, \fa \al \in I$\\
$2)\; T_{\al}(x_0)=x_0, \fa \al\in I$;\\
$3)\; T_{\al}\circ T_{\be}=T_{\al\be}, \fa \al,\be\in I$; and\\
$4)\; \lim_{\al\ri 1} T_{\al}(x)$ exists $\fa x\in X$\\ \end{defi}

We are mainly concerned with dilation families where the index set is a 'pure' index set, as described below:

\begin{defi} Let $I$ be a subset of $\ma R_{\geq 0}$ closed under multiplication containing $1$. If - $1)\; I\subset [1,\nf)$ and for any $\al<\be\in I, \frac{\be}{\al}\in I$; or $2)\; I\subset [0,1]$ and for any $\al<\be\in I, \frac{\al}{\be}\in I$; or $3)\; I$ is closed under nonzero division - then $I$ is called a pure set.\\ \end{defi}

 Regarding the $4$th condition in the definition of a dilation family — this is a, seemingly weak, kind of continuity requirement on the dilation family. As we will see, these structures represent a strengthening of classical contractibility conditions that appear in general topology. However, they may also be relevant to the study of certain kinds of metric measures, such as many Hausdorff measures, in which a dilation of a measurable set by scale $\al$ yields a measurable set with a measure scaled by a factor of $\al^d$ ($d$ being the dimension of the measure). It also allows one to define (once suitably extended) differentiability of a function at the 'center' of the dilation family.\\

 We will start off with some results that just establish basic properties of dilation families and necessary conditions for their existence.\\

\begin{prop}
\label{prop: if 1 or 0 is in I} Let $(X,d)$ be a metric space. Let $\{T_{\al}\}_{\al\in I}$ be a dilation family on $(X,d,x_0)$. Then $T_1=\tr{Id}_X$. If $0\in I$, $T_0=O_{x_0}$.\end{prop}

\begin{proof} For any $x\in X$, we have $T_1(T_1(x))=T_{1^2}(x)=T_1(x)$, so $d(x,T_1(x))=1\cd d(x,T_1(x))=d(T_1(x),T_1(T_1(x)))=0$, i.e. $x=T_1(x), \fa x\in X$. Also, we must have $T_0(x_0)=x_0$, but for any $x\in X$, we have $d(T_0(x),T_0(x_0))=0\cd d(x,x_0)=0$, so $T_0(x)=x_0,\fa x\in X$.\\ \end{proof}

\begin{prop} 
\label{prop: limit is identity} Let $(X,d)$ be a metric space. Let $\{T_{\al}\}_{\al\in I}$ be a dilation family on $(X,d,x_0)$. For any $x\in X$, we have $\lim_{\al\ri 1} T_{\al}(x)=x$.\end{prop}

\begin{proof} Let $x$ be any point of $X$. Let $y$ denote $\lim_{\al \ri 1} T_{\al}(x)$. Then we also have that $y=\lim_{\al\ri 1} T_{\al^2}(x)$.\\

This means 

$$0=d(y,y)=d(\lim_{\al\ri 1} T_{\al^2}(x),\lim_{\al\ri 1} T_{\al}(x))$$

$$=\lim_{\al\ri 1} d(T_{\al^2}(x),T_{\al}(x))=\lim_{\al\ri 1} d(T_{\al}(T_{\al}(x)),T_{\al}(x))$$

$$=\lim_{\al\ri 1} \al d(T_{\al}(x),x)=\lim_{\al \ri 1} d(T_{\al}(x),x)$$

$$=d(\lim_{\al \ri 1} T_{\al}(x),\lim_{\al \ri 1} x)=d(y,x)$$

It follows that $y=x$.\\ \end{proof}

\begin{prop} 
\label{prop: x-function is cont} Let $(X,d)$ be a metric space. Let $\{T_{\al}\}_{\al\in I}$ be a dilation family on $(X,d,x_0)$. Then, for any $x\in X$, the function $T^x: I\ri (X,d)$ defined as $\al \mapsto T_{\al}(x)$ is continuous and uniformly continuous on any set $[a,b]\cap I$ where $a>0$.\end{prop}

\begin{proof} Let any $\al\in I$ be given. We show continuity of $T^x$ at $\al$. There are three cases depending on the type of pure set $I$ is - these break down into three cases for $\al$, whether it's in $(0,1]$, $[1,\nf)$, or equal to $0$.\\

Suppose $\al$ is positive. Now, let any $\eps\in\ma R^+$ be given. Since $\lim_{\be \ri 1} T_{\be}(x)$ exists and equals $x$, there exists $0<\de<1$ such that whenever $\be \in B(1;\de)\cap I$ we have $d(T_{\be}(x),x)<\frac{\eps}{2\al}$. Now suppose $\be$ is any member of the set $B(\al; \frac{\al \de}{2})\cap I$. If $\al \in (0,1]$ and $\be<\al$, then we have:

$$d(T^x(\al),T^x(\be))=d(T_{\al}(x),T_{\be}(x))$$

$$=d(T_{\al}(x),T_{\al}(T_{\frac{\be}{\al}}(x)))$$

$$=\al d(x,T_{\frac{\be}{\al}}(x))$$

As $|\be-\al|<\frac{\al \de}{2}$, $|\frac{\be}{\al}-1|<\frac{\de}{2}<\de$, so we have $d(T^x(\al),T^x(\be))<\al \cd \frac{\eps}{2\al}<\eps$.\\

If $\al\in (0,1]$ and $\be>\al$, then we have:

$$d(T^x(\al),T^x(\be))=d(T_{\al}(x),T_{\be}(x))$$

$$=d(T_{\be}(T_{\frac{\al}{\be}}(x)),T_{\be}(x))$$

$$=\be d(T_{\frac{\al}{\be}}(x),x)$$

As $|\be-\al|<\frac{\al \de}{2}$, $|1-\frac{\al}{\be}|<\frac{\al}{\be}\frac{\de}{2}<\frac{\de}{2}<\de$, so we have $d(T^x(\al),T^x(\be))<\be \cd \frac{\eps}{2\al}$. As $0<\de<1$, $0<\be<\frac{3\al}{2}<2\al$, so this means $d(T^x(\al),T^x(\be))<\eps$.\\

If $\al\in [1,\nf)$ and $\be<\al$, then we have:

$$d(T^x(\al),T^x(\be))=d(T_{\al}(x),T_{\be}(x))$$

$$=d(T_{\be}(T_{\frac{\al}{\be}}(x)),T_{\be}(x))$$

$$=\be d(T_{\frac{\al}{\be}}(x),x)$$

As $|\be-\al|<\frac{\al \de}{2}$, $|1-\frac{\al}{\be}|<\frac{\al}{\be}\frac{\de}{2}$. As $\de<1$, $\be>\frac{\al}{2}$, so $|1-\frac{\al}{\be}|<\de$, giving us $d(T^x(\al),T^x(\be))<\be \cd \frac{\eps}{2\al}$. As $0<\de<1$, $0<\be<\frac{3\al}{2}<2\al$, so this means $d(T^x(\al),T^x(\be))<\eps$.\\

And if $\al\in [1,\nf)$ and $\be>\al$, then we have:

$$d(T^x(\al),T^x(\be))=d(T_{\al}(x),T_{\be}(x))$$

$$=d(T_{\al}(x),T_{\al}(T_{\frac{\be}{\al}}(x)))$$

$$=\al d(x,T_{\frac{\be}{\al}}(x))$$

As $|\be-\al|<\frac{\al \de}{2}$, $|\frac{\be}{\al}-1|<\frac{\de}{2}<\de$, so we have $d(T^x(\al),T^x(\be))<\al \cd \frac{\eps}{2\al}<\eps$.\\

So for any $\al>0$, we have $T^x(B(\al;\frac{\al\de}{2})\cap I)\subset B(T^x(\al);\eps)$. But this means that for any set $[a,b]\cap I$ with $a>0$ and any $\al\in [a,b]\cap I$, we have $T^x(B(\al;\frac{a\de}{2})\cap I)\subset T^x(B(\al;\frac{\al\de}{2})\cap I)\subset B(T^x(\al);\eps)$. As $\eps$ was an arbitrary positive real, we have that $T^x$ is uniformly continuous on any set $[a,b]\cap I$, and so is also continuous on $(0,\nf)\cap I$.\\

Now, suppose $0\in I$. To prove continuity of $T^x$ at $0$, again let any $\eps>0$ be given. If $x=x_0$, then $T^x$ is the constant function $x_0$, and so $T^x$ is trivially continuous at $0$. Assuming $x\neq x_0$, then given any $\be\in B(0;\frac{\eps}{d(x_0,x)})$, we have:

$$d(T^x(0),T^x(\be))=d(T_0(x),T_{\be}(x))$$

$$=d(x_0,T_{\be}(x))=d(T_{\be}(x_0),T_{\be}(x))$$

$$=\be d(x_0,x)<\frac{\eps}{d(x_0,x)}\cd d(x_0,x)=\eps$$

making $T^x$ continuous at $0$ as well.\\ \end{proof}

\begin{prop}
\label{prop: extend dilation family} Let $(X,d)$ be a Cauchy metric space. Let $\{T_{\al}\}_{\al\in I}$ be a dilation family on $(X,d,x_0)$. Then, setting $J=\overline{I}$ (where the closure is taken in $\ma R_{\geq 0}$), $J$ is also a pure set and there exists a dilation family about $x_0$, $\{T_{\al}\}_{\al\in J}$, that extends $\{T_{\al}\}_{\al\in I}$.\end{prop}

\begin{proof} First, I show that $J=\overline{I}$ is pure. As $I$ contains $1$, $J$ does too.\\

If $I\subset [0,1]$, then $J\subset [0,1]$ as well. Given any $c<d\in J$, we can choose sequences of elements, $\{\al_i\}_{i=1}^{\nf}\subset I$ converging to $c$ and $\{\be_i\}_{i=1}^{\nf}\subset I$ converging to $d$, such that $\al_i<\be_i, \fa i\in \ma Z^+$. Then $\frac{\al_i}{\be_i}\in I, \fa i\in\ma Z^+$, and $\lim_{i\ri \nf} \frac{\al_i}{\be_i}=\frac{c}{d}$, so $\frac{c}{d}\in J$, making $J$ pure.\\

If $I\subset [1,\nf)$, then $J\subset [1,\nf)$ as well. Given any $c<d\in J$, we can choose sequences of elements, $\{\al_i\}_{i=1}^{\nf}\subset I$ converging to $c$ and $\{\be_i\}_{i=1}^{\nf}\subset I$ converging to $d$, such that $\al_i<\be_i, \fa i\in \ma Z^+$. Then $\frac{\be_i}{\al_i}\in I, \fa i\in\ma Z^+$, and $\lim_{i\ri \nf} \frac{\be_i}{\al_i}=\frac{d}{c}$, so $\frac{d}{c}\in J$, making $J$ pure.\\

If $I$ intersects both $[0,1)$ and $[1,\nf)$, then $I$ closed under nonzero division, so $J$ is as well, making $J$ a pure set as well.\\

Now, using Proposition \ref{prop: x-function is cont} (in particular, the fact that $T^x$ is uniformly continuous on $[a,b]$ for $a>0$), for any $\al\in J$, define a function $T_{\al}\in X^X$ such that $T_{\al}(x)=\lim_{\substack{\be\ri \al,\\ \be\in I}} T^x(\be)$. The extension, $T^x:J\ri X$, defined as $T^x(\al)=T_{\al}(x)$ is continuous.\\

Now, for any $\al\in J$, we have 

$$T_{\al}(x_0)=\lim_{\substack{\be\ri \al,\\ \be\in I}} T^x(\be)$$

$$=\lim_{\substack{\be\ri \al,\\ \be\in I}} T_{\be}(x_0)=\lim_{\substack{\be\ri \al,\\ \be\in I}} x_0=x_0$$

Also $T_1(x_0)=x_0$.\\

Next, for any $\al\in \overline{I}$ and any $x,y\in X$, we have 

$$d(T_{\al}(x),T_{\al}(y))=d(\lim_{\substack{\be\ri \al,\\ \be\in I}} T^x(\be),\lim_{\substack{\be\ri \al,\\ \be\in I}} T^y(\be))$$

$$=\lim_{\substack{\be\ri \al,\\ \be\in I}} d( T_{\be}(x),T_{\be}(y))$$

$$=\lim_{\substack{\be\ri \al,\\ \be\in I}} \be d(x,y)$$

$$=\al d(x,y)$$

so $T_{\al}$ is a dilation of scale $\al, \fa \al\in J$.\\

Given any $\al,\be\in \overline{I}$ and any $x\in X$ (appealing to the continuity of $T_{\al}$), we have:

$$(T_{\al}\circ T_{\be})(x)=T_{\al}(T_{\be}(x))=T_{\al}(\lim_{\substack{\ga\ri \be,\\ \ga\in I}} T^x(\ga))$$

$$=\lim_{\substack{\ga\ri \be,\\ \ga\in I}} T_{\al}(T^x(\ga))=\lim_{\substack{\ga\ri \be,\\ \ga\in I}} T_{\al}(T_{\ga}(x))$$

$$=\lim_{\substack{\ga\ri \be,\\ \ga\in I}} \lim_{\substack{\de\ri \al,\\ \de\in I}} T^{T_{\ga}(x)}(\de)=\lim_{\substack{\ga\ri \be,\\ \ga\in I}} \lim_{\substack{\de\ri \al,\\ \de\in I}} T_{\de}(T_{\ga}(x))$$

$$=\lim_{\substack{\ga\ri \be,\\ \ga\in I}} \lim_{\substack{\de\ri \al,\\ \de\in I}} T_{\de\ga}(x)=\lim_{\substack{\de\ga\ri \al\be,\\ \de,\ga\in I}} T^x(\de\ga)$$

$$=T_{\al\be}(x)$$

where the last two equalities hold because $T^x: J \ri X$ is continuous.\\

As $x$ was an arbitrary point of $X$, this means $T_{\al}\circ T_{\be}=T_{\al\be}, \fa \al,\be\in J$.\\

Lastly, as $T^x:J\ri X$ is continuous $\fa x\in X$, we have for any $x\in X$ that $\lim_{\al\ri 1} T_{\al}(x)$ exists. So $\{T_{\al}\}_{\al\in J}$ is a dilation family about $x_0$ extending $\{T_{\al}\}_{\al\in I}$, as desired.\\ \end{proof}

\begin{prop}
\label{prop: adding 0} Let $(X,d)$ be a metric space. Let $\{T_{\al}\}_{\al\in I}$ be a dilation family on $(X,d,x_0)$. Suppose $I\cap [0,1)\neq \emptyset$. Then $J=I\cup \{0\}$ is pure and, setting $T_0=O_{x_0}$, the collection of functions $\{T_{\al}\}_{\al\in J}$ is also a dilation family on $(X,d,x_0)$.\end{prop}

\begin{proof} Supposing $I\cap [0,1)\neq \emptyset$, let $J=I\cup\{0\}$.\\

If $I\subset [0,1]$, then $J\subset [0,1]$ as well. Given any $\al<\be\in K$, either $\al,\be\in I$, in which case (as $I$ as pure), $\frac{\al}{\be}\in J$. If $0<\al$, then $\frac{0}{\al}=0\in K$. So $J$ is pure.\\

If $I\not\subset [0,1]$, then $I$ is closed under nonzero division. $J$ will then also be closed under nonzero division. And given any $\al\neq 0\in K$, we have $\frac{0}{\al}=0\in J$, so it follows that $J$ is closed under nonzero division as well, i.e. $J$ is pure too.\\

Finally, we set $T_0$ equal the constant map $O_{x_0}$. By Proposition \ref{prop: if 1 or 0 is in I}, this definition produces no inconsistency. Given any $\al\in K$, we have $T_{\al}\circ T_0=T_{\al}\circ O_{x_0}=O_{x_0}=O_{x_0}\circ T_{\al}=T_0\circ T_{\al}$. So $\{T_{\al}\}_{\al\in K}$ is also a dilation family on $(X,d,x_0)$, completing the proof.\\ \end{proof}

We now establish one of the central claims made earlier about dilation families - that they are a 'a strengthening of classical contractibility conditions' (specifically, dilation families with index sets that cover at least $[0,1]$):

\begin{prop}
\label{prop: continuous action} Let $(X,d)$ be a metric space and $x_0$ be a point in $X$. If $\{T_{\al}\}_{\al\in I}$ is a dilation family on $(X,d,x_0)$, then the function $F:I\ti X\ri X$ defined as $F(\al,x)=T_{\al}(x), \fa (\al,x)\in I\ti X$ is continuous.\end{prop}

\begin{proof} Let $(\al,x)$ be any point of $I\ti X$ and let $\eps$ be any positive real number. By Proposition \ref{prop: x-function is cont}, we can choose some $\de\in (0,\al)$ such that whenever $\be\in B(\al;\de)$, $d(T^x(\al),T^x(\be))<\frac{\eps}{2}$. Now, if $\al\geq 1$, for any $(\be,y)\in B(\al;\de)\ti B_d(x;\frac{\eps}{4\al})$, we have:

$$d(F(\al,x),F(\be,y))=d(T_{\al}(x),T_{\be}(y))$$

$$=d(T_{\al}(x),T_{\be}(x))+d(T_{\be}(x),T_{\be}(y))$$

$$=d(T^x(\al),T^x(\be))+\be d(x,y)$$

$$<\frac{\eps}{2}+\be\cd \frac{\eps}{4\al}=\frac{\eps}{2}+\frac{\eps}{2}=\eps$$

where the second-last equality follows from the fact that (as $\be\in B(\al;\de)$ and $0<\de<\al$) $\be<2\al$.\\

If $\al<1$, then for any $(\be,y)\in B(\al;\de)\ti B_d(x;\frac{\eps}{4})$, we have:

$$d(F(\al,x),F(\be,y))=d(T_{\al}(x),T_{\be}(y))$$

$$=d(T_{\al}(x),T_{\be}(x))+d(T_{\be}(x),T_{\be}(y))$$

$$=d(T^x(\al),T^x(\be))+\be d(x,y)$$

$$<\frac{\eps}{2}+\be\cd \frac{\eps}{4}<\frac{\eps}{2}+2\al\cd \frac{\eps}{4}$$

$$<\frac{\eps}{2}+\frac{\eps}{2}=\eps$$

As $(\al,x)$ was an arbitrary point of $I\ti X$ and $\eps$ was an arbitrary positive real, $F$ is continuous everywhere, as desired.\\ \end{proof}

\begin{co}
\label{co: continuous map into function space} Let $(X,d)$ be a metric space and $x_0$ be a point in $X$. If $\{T_{\al}\}_{\al\in I}$ is a dilation family on $(X,d,x_0)$, then the function $F:I\ri C(X)$ defined as $F(\al)(x)=T_{\al}(x), \fa (\al,x)\in I\ti X$ is continuous ($C(X)$ has the compact-open topology). \end{co}

\begin{proof} This follows directly from Proposition \ref{prop: continuous action} and the fact that the induced function of a continuous function produces a continuous function into the function space (provided the function space is equipped with the compact-open topology).\\ \end{proof}

\begin{defi} Let $X$ be a topological space. The symbol $C_{x_0}(X)$ denotes the set $\{f\in C(X) | f(x_0)=x_0\}$. \\ \end{defi}

We mostly work with the compact-open topology (on $C_{x_0}(X)$, we usually apply the subspace topology induced by the compact-open topology).\\

The rest of this paper will be devoted to finding sufficient conditions for the existence of dilation families. The way we will do this is by first demonstrating that all metrizable cone spaces can be given a metric under which they possess a very special kind of dilation family and then describing conditions under which topological or metric spaces may be converted into such metrizable cone spaces. These 'special kind' of dilation families are the following:

\begin{defi} Let $(X,d)$ be a metric space and $x_0$ be a point in $X$. If $\{T_{\al}\}_{\al\in I}$ is a dilation family on $(X,d,x_0)$, and we have for any $\al\leq \be\leq \ga\in I$, $d(T_{\al}(x),T_{\ga}(x))=d(T_{\al}(x),T_{\be}(x))+d(T_{\be}(x),T_{\ga}(x)), \fa x\in X$, then we say that $\{T_{\al}\}_{\al\in I}$ is a linear dilation family on $(X,d,x_0)$.\\ \end{defi}

\begin{defi} Let $C$ be a topological space, $I$ be a subset of $\ma R$ containing $0$, and $p$ be any point. $\con_p(I,C)$ denotes the cone formed from $I\ti C$ by identifying $\{0\}\ti C$ with $p$.\\ \end{defi}

The following result provides us with a general method of producing metrics for metrizable cones under which they have linear dilation families:\\

\begin{prop}
\label{prop: metrizable cones} Let $(C,d)$ be a metric space with diameter at most $2$. Letting $X$ denote $\con_{x_0}([0,1],C)$, the function $D:X\ti X\ri \ma R_{\geq 0}$ defined as $D(x_0,x_0)=0, D((a,c_1),x_0)=D(x_0,(a,c_1))=a$ and $D((a,c_1),(b,c_2))=|a-b|+\min(a,b)d(c_1,c_2), \fa a,b\in (0,1],\fa c_1,c_2\in C$ is a metric on $X$ preserving the topology of $X$. Also, the family of functions $\{F(\al)\}_{\al\in [0,1]}$ defined as $F(\al)(x_0)=x_0, F(\al)(a,x)=(a\al,x), \fa \al\in (0,1]$, and $F(0)=O_{x_0}$ is a linear dilation family on $(X,D,x_0)$. Furthermore, $\tr{Bd}(B_D(x_0;1))=\{1\}\ti C$.\end{prop}

\begin{proof} 

\tb{\underline{Part 1:}}\\

Clearly $D$ is a well-defined nonnegative symmetric function.\\

We have $D(x_0,x_0)=0$ and for any $(a,c)\in X-\{x_0\}$, $D((a,c),(a,c))=|a-a|+\min(a,a)d(c,c)=0+0=0$.\\

Conversely, if $D(z_1,z_2)=0$, there are two cases, the first of which is where $z_1,z_2\neq x_0$ in which case $z_1=(a,c_1), z_2=(b,c_2)$ for some $a,b\in (0,1]$, and some $c_1,c_2\in C$. This means $|a-b|+\min(a,b)d(c_1,c_2)=0\Ri |a-b|=0, \min(a,b)d(c_1,c_2)=0$. As $a,b>0$, $\min(a,b)>0$, so we get $d(c_1,c_2)=0$ as well, i.e. $a=b$ and $c_1=c_2$, giving us that $z_1=(a,c_1)=(b,c_2)=z_2$. The other case (WLOG) is where $z_1=x_0$. If $z_2\neq x_0$, then $z_2=(a,c)$ for some $a\in (0,1]$ and some $c\in C$, so $0=D(z_1,z_2)=D(x_0,(a,c))=a$, a contradiction, so we must have $z_2=x_0=z_1$. By the symmetry of $D$, the same occurs if $z_2=x_0$. So $D$ satisfies the identity of indiscernibles.\\

Now we must show that $D$ satisfies the triangle inequality, i.e. $D(x,z)\leq D(x,y)+D(y,z), \fa x,y,z\in X$. The first case is where $x,y,$ and $z$ are all not equal to $x_0$. We can write $x=(a,c_1), y=(b,c_2),$ and $z=(c,c_3)$ for some $a,b,c\in (0,1]$ and some $c_1,c_2,c_3\in C$. We assume WLOG that $a\geq c$. There are two subcases - the first is where $b\geq c$ and the second is where $b\leq c$. If $b\geq c$, then:

$$D(x,z)=|a-c|+\min(a,c)d(c_1,c_3)$$

$$=|a-c|+c d(c_1,c_3)$$

$$\leq |a-b|+|b-c|+c d(c_1,c_3)$$

$$\leq |a-b|+|b-c|+c d(c_1,c_2)+c d(c_2,c_3)$$

$$\leq |a-b|+|b-c|+\min(a,b) d(c_1,c_2)+\min(b,c) d(c_2,c_3)$$

$$=D(x,y)+D(y,z)$$

If $b\leq c$, then

$$D(x,z)=|a-c|+\min(a,c)d(c_1,c_3)$$

$$=a-c+c d(c_1,c_3)$$

Now 

$$[a+c-2b+b d(c_1,c_3)]-[a-c+c d(c_1,c_3)]$$

$$=2c-2b+(b-c)d(c_1,c_3)$$

$$=(c-b)(2-d(c_1,c_3))$$

$$\geq 0$$

where the final inequality follows as $d$ is bounded above by $2$. So:

$$D(x,z)=a-c+c d(c_1,c_3)$$

$$\leq a+c-2b+b d(c_1,c_3)$$

$$=(a-b)+(c-b)+b d(c_1,c_3)$$

$$=|a-b|+|b-c|+b d(c_1,c_3)$$

$$\leq |a-b|+|b-c|+b d(c_1,c_2)+b d(c_2,c_3)$$

$$= |a-b|+|b-c|+\min(a,b) d(c_1,c_2)+\min(b,c) d(c_2,c_3)$$

$$= D(x,y)+D(y,z)$$

The next case is where $z=x_0$, and $x,y\neq x_0$. We have $x=(a,c_1), y=(b,c_2)$ for some $a,b\in (0,1]$, and some $c_1,c_2\in C$. There are again two subcases here - where $a\geq b$ and $a\leq b$.\\

If $\al\geq \be$, then:

$$D(x,z)=D(x,x_0)=a$$

$$=b+(a-b)$$

$$\leq b+|a-b|+\min(a,b)d(c_1,c_2)$$

$$=D(x,y)+D(y,x_0)=D(x,y)+D(y,z)$$

If $a\leq b$, then:

$$D(x,z)=D(x,x_0)=a\leq b$$

$$\leq b +|a-b|+\min(a,b)d(c_1,c_2)$$

$$=D(x,y)+D(y,x_0)=D(x,y)+D(y,z)$$

And the next case (and the last nontrivial case) is where $y=x_0$ but $x,z\neq x_0$. We can write $x=(a,c_1), z=(c,c_3)$ for some $a,c\in (0,1]$, and some $c_1,c_3\in C$. Here (again, assuming WLOG that $a\geq c$), we have:

$$D(x,z)=|a-c|+\min(a,c) d(c_1,c_3)$$

$$=a-c + c d(c_1,c_3)$$

$$\leq a-c+2c=a+c$$

$$=D(x,x_0)+D(x_0,z)=D(x,y)+D(y,z)$$

These cover all nontrivial cases (by appropriately also applying the symmetry property of $D$). The remaining cases are just those where two or more of $x,y,$ and $z$ equal $x_0$, but these are all trivial.\\

If $y=x_0=z$, then $D(y,z)=0$, so $D(x,z)\leq D(x,z)+0=D(x,z)+D(y,z)=D(x,y)+D(y,z)$. Similarly, if $x=z_0=z$, then $D(x,z)=0$ which, by nonnegativity of $D$, is automatically $\leq D(x,y)+D(y,z)$. So $D$ satisfies the triangle inequality, making it a metric extending $d$.\\

\tb{\underline{Part 2:}}\\

Next, we show that the topology induced by $D$ is that of $\con_{x_0}([0,1],C)$.\\

First, let $O$ be any open set of $\con_{x_0}([0,1],C)$ and $x$ be any member of $O$. If $x$ equals some $(a,c)\in (0,1]\ti C$, then we may choose some $\eps>0$ such that $x=(a,c_1)\in B(a;\eps)\ti B_d(c_1;\eps)\subset O$ (where $B(a;\eps)$ denotes the open ball of radius $\eps$ about $a$ in $(0,1]$). We may assume WLOG that $\eps<\frac{a}{2}$. Now let $y$ be any member of $B_D\left(x;\frac{a\eps}{2}\right)$. If $y=x_0$, then $a=D(x_0,(a,c_1))=D(y,x)<\frac{a\eps}{2}<\frac{a^2}{2}<a$, a contradiction, so $y$ equals some $(b,c_2)\in (0,1]\ti C$. Since $y\in B_D\left(x;\frac{a\eps}{2}\right)$, we have:

$$|a-b|\leq |a-b|+\min(a,b)d(c_1,c_2)$$

$$=D(x,y)<\frac{a\eps}{2}<\eps$$

so $b\in B(a;\eps)$. Also, as $\eps<\frac{a}{2}$, $b\in B\left(a;\frac{a}{2}\right)$, which means $b>\frac{a}{2}$. This means:

$$\frac{a}{2}d(c_1,c_2)<\min(a,b)d(c_1,c_2)$$

$$\leq |a-b|+\min(a,b)d(c_1,c_2)$$

$$=D(x,y)<\frac{a\eps}{2}$$

Cancelling $\frac{a}{2}$ from the first and last expressions, we get that $d(c_1,c_2)<\eps$, i.e. $c_2\in B_d(c_1;\eps)$. So $y=(b,c_2)\in B(a;\eps)\ti B_d(c_1;\eps)\subset O$. As $y$ was an arbitrary point of $B_D\left(x;\frac{a\eps}{2}\right)$, we have $x\in B_D\left(x;\frac{a\eps}{2}\right)\subset O$.\\

If instead we have $x=x_0$, then for some $a\in (0,1]$, we have that $\{x_0\}\cup (0,a)\ti C\subset O$. But this former set is easily seen to just be $B_D(x_0;a)$, so $x_0\in B_D(x_0;a)\subset O$. It follows that the topology induced by $D$ is at least as fine as that of $\con_{x_0}([0,1],C)$.\\

Conversely, let $O$ be any open set of $X$ under $D$ and $x$ be any member of $O$. Then we can choose some $\eps\in \ma R^+$ such that $B_D(x;\eps)\subset O$. If $x$ equals some $(a,c_1)\in (0,1]\ti C$, then let $y=(b,c_2)$ be any element of $B\left(a;\frac{\eps}{2}\right)\ti B_d\left(c_1;\frac{\eps}{2a}\right)$ (where the first open ball is taken in $(0,1]$).\\ 

$$D(x,y)=|a-b|+\min(a,b)d(c_1,c_2)$$

$$<\frac{\eps}{2}+\min(a,b)\frac{\eps}{2a}$$

$$\leq \frac{\eps}{2}+a\ti \frac{\eps}{2a}=\eps$$

so $y\in B_D(x;\eps)$, i.e. $x\in B\left(a;\frac{\eps}{2}\right)\ti B_d\left(c_1;\frac{\eps}{2a}\right)\subset B_D(x;\eps)$.\\

If instead $x=x_0$, then $B_D(x;\eps)=\{x_0\}\cup (0,\eps)\ti C$. It follows that the topology of $\con_{x_0}([0,1],C)$ is at least as fine as that induced by $D$ as well, so that the two topologies are the same.\\

\tb{\underline{Part 3:}}\\

Now, for any $\al\in (0,1]$ and any $x,y\in X-\{x_0\}$ we can write $x=(a,c_1), y=(b,c_2)$ for some $a,b\in (0,1]$ and some $c_1,c_2\in C$ we have:

$$D(F(\al)(x),F(\al)(y))=D(F(\al)(a,c_1),F(\al)(b,c_2))$$

$$=D((\al a,c_1),(\al b,c_2))$$

$$=|\al a-\al b|+\min(\al a,\al b)d(c_1,c_2)$$

$$=\al(|a-b|+\min(a,b)d(c_1,c_2))$$

$$=\al D((a,c_1),(b,c_2))=\al D(x,y)$$

and

$$=D(F(\al)(x),F(\al)(x_0))=D(F(\al)(a,c_1),x_0)$$

$$=D((\al a,c_1),x_0)$$

$$=\al a=\al D((a,c_1),x_0)$$

$$=\al D(x,x_0)$$

so $F(\al)$ is a dilation of scale $\al$, $\fa \al\in (0,1]$. Since $F(0)$ is a constant map, it follows that $F(\al)$ is dilation of scale $\al$, $\fa \al\in [0,1]$.\\

Also, given any $\al,\be\in (0,1]$ and any $(a,c)\in (0,1]\ti C$ we have:

$$(F(\al)\circ F(\be))(a,c)=F(\al)(\be a,c)$$

$$=((\al\be)a,c)$$

$$=F(\al\be)(a,c)$$

and $(F(\al)\circ F(\be))(x_0)=F(\al)(x_0)=x_0=F(\al\be)(x_0)$, so $F(\al)\circ F(\be)=F(\al\be), \fa \al,\be\in (0,1]$. Also, $F(\al)$ fixes $x_0, \fa \al\in [0,1]$, so:

$$F(\al)\circ F(0)=F(\al)\circ O_{x_0}$$

$$=O_{x_0}=F(\al\ti 0)=F(0\ti\al)$$

$$=O_{x_0}=O_{x_0}\circ F(\al)$$

$$=F(0)\circ F(\al)$$

so $F(\al)\circ F(\be)=F(\al\be), \fa \al,\be \in [0,1]$, meaning $\{F(\al)\}_{\al\in [0,1]}$ is a dilation family on $(X,D,x_0)$.\\

Next, let any $\al\leq \be\leq \ga\in [0,1]$ be given. For any $(a,c)\in (0,1]\ti C$, if $\al>0$, we have:

$$D(F(\al)(a,c),F(\ga)(a,c))=D((\al a,c),(\ga a, c))$$

$$=|\al a-\ga a|+d(c,c)$$

$$=(\ga-\al)a=(\ga-\be+\be-\al)a$$

$$=|\al-\be|a+|\be-\ga|a$$

$$=|\al a-\be a|+|\be a-\ga a|$$

$$=(|\al a-\be a|+d(c,c))+(|\be a-\ga a|+d(c,c))$$

$$=D((\al a,c),(\be a,c))+D((\be a,c),(\ga a,c))$$

$$=D(F(\al)(a,c),F(\be)(a,c))+D(F(\be)(a,c),F(\ga)(a,c))$$

and if $\al=0$:

$$D(F(\al)(a,c),F(\ga)(a,c))=D(x_0,(\ga a, c))$$

$$=\ga a=\be a+|\be a-\ga a|$$

$$=D(x_0,(\be a,c))+(|\be a-\ga a|+d(c,c))$$

$$=D(F(\al)(a,c),F(\be)(a,c))+D((\be a,c),(\ga a,c))$$

$$=D(F(\al)(a,c),F(\be)(a,c))+D(F(\be)(a,c),F(\ga)(a,c))$$

(the cases where $\be$ or $\ga$ are $0$ are trivial). Also, for any $\al\leq \be\leq \ga\in [0,1]$, we have: 

$$D(F(\al)(x_0),F(\ga)(x_0))=D(x_0,x_0)$$

$$=0=0+0=D(x_0,x_0)+D(x_0,x_0)$$

$$=D(F(\al)(x_0),F(\be)(x_0))+D(F(\be)(x_0),F(\ga)(x_0))$$

So $\{F(\al)\}_{\al\in [0,1]}$ is in fact a linear dilation family on on $(X,D,x_0)$.\\

Lastly, for any $x\in \con_{x_0}([0,1],C)$, we have:

$$x\in \tr{Bd}(B_D(x_0;1))\iff D(x,x_0)=1$$

$$\iff x=(1,c), c\in C\iff x\in \{1\}\ti C$$

so $\tr{Bd}(B_D(x_0;1))=\{1\}\ti C$, completing the proof.\\ \end{proof}

Now, we demonstrate a condition for compact spaces that resembles that of the existence of dilation families that 'contract' and is strong enough to ensure they are metrizable cones (note that dilation families that 'expand' are not possible over compact spaces).\\ 

\begin{prop}
\label{prop: compact space cone condition} Let $(X,\tau)$ be a compact Hausdorff topological space. Suppose for some $x_0\in X$, there exists a continuous monoidal monomomorphism $F: ([0,1],\ti) \ri (C_{x_0}(X),\circ)$ such that: $1) F(0)=O_{x_0}$, $2) F(1)=\tr{Id}_X$, and $3) F(\al)$ is injective $\fa \al\in (0,1]$. Then, with $C=X-F[0,1)(X)$, $\{F(\al)(C)\}_{\al\in [0,1]}$ is a partition of $X$. If $F[0,1)(X)$ is an open subset of $X$, then the function $f':\con_{x_0}([0,1],C)\ri X$ defined so that $x_0\mapsto x_0$ and $(\al,c)\mapsto F(\al)(c)$ is a homeomorphism.\\
\end{prop}

\begin{proof} 

As $F$ is a monomorphism, $F(0)\neq F(1)$, i.e. $O_{x_0}\neq \tr{Id}_X$. This means $X-\{x_0\}$ is non-empty.\\

As $X$ is compact and $F: ([0,1],\ti)\ri (C_{x_0}(X),\circ)$ is continuous, the function $f: [0,1]\ti X\ri X$ defined as $f(\al,x)=F(\al)(x), \fa (\al,x)\in [0,1]\ti X$ is continuous as well. For any $x\in X$ then (as $X$ is Hausdorff), $f^{-1}(\{x\})$ is a closed subset of $[0,1]\ti X$. Letting $\pi_1:[0,1]\ti X\ri [0,1]$ denote the projection map onto the first component, as $X$ is compact, $\pi_1$ is a closed map, so $\pi_1(f^{-1}(\{x\}))$ is a closed subset of $[0,1], \fa x\in X$. As $[0,1]$ is compact, this means that $\pi_1(f^{-1}(\{x\}))$ always has a least element - let $\Ga: X\ri [0,1]$ be the function that gives this least element for each $x$. As $F(0)=O_{x_0}$, $\Ga(x_0)=0$ and $\Ga(x)>0, \fa x\in X-\{x_0\}$.\\

Now, by the definition of $\Ga$, for any $x\in X-\{x_0\}$, we have:

$$x\in F(\Ga(x))(X)-\cup_{\al\in [0,\Ga(x))} F(\al)(X)$$

$$=F(\Ga(x))(X)-\cup_{\be\in [0,1)} F(\Ga(x)\be)(X)$$

$$=F(\Ga(x))(X)-\cup_{\be\in [0,1)} (F(\Ga(x))\circ F(\be))(X)$$

$$=F(\Ga(x))(X-\cup_{\be\in [0,1)} F(\be)(X))$$

$$=F(\Ga(x))(X-F[0,1)(X))=F(\Ga(x))(C)$$

where the third equality holds as $\Ga(x)>0$ and $F(\al)$ is injective $\fa \al\in (0,1]$.\\

As $x$ was an arbitrary point of $X-\{x_0\}$, the collection $\{F(\al)(C)\}_{\al\in (0,1]}$ covers $X-\{x_0\}$. If $F(\al)(C)$ is empty for any $\al\in (0,1]$, then $C$ itself will be empty so that $X-\{x_0\}=\cup_{\al\in (0,1]} F(\al)(C)=\cup_{\al\in (0,1]} F(\al)(\emptyset)=\cup_{\al\in (0,1]} \emptyset=\emptyset$, a contradiction. So each element of the collection $\{F(\al)(C)\}_{\al\in (0,1]}$ is nonempty. In particular, $C=F(1)(C)$ itself is nonempty.\\

As $F(0)=O_{x_0}$ and $C$ is nonempty, $F(0)(C)=\{x_0\}$, so $\{F(\al)(C)\}_{\al\in [0,1]}$ is a covering of $X$ by nonempty sets. Lastly, for any $\al>\be\in [0,1]$, if $x\in F(\al)(C)\cap F(\be)(C)$, then:

$$x\in F(\al)(C)=F(\al)(X-F[0,1)(X))$$

$$=F(\al)(X)-F(\al)(F[0,1)(X))$$

$$=F(\al)(X)-F[0,\al)(X)$$

$$\subset F(\al)(X)-F(\be)(X)$$

$$\subset F(\al)(X)-F(\be)(C)$$

a contradiction. It follows that all sets in $\{F(\al)(C)\}_{\al\in [0,1]}$ are pairwise disjoint. Hence, $\{F(\al)(C)\}_{\al\in [0,1]}$ is a partition of $X$.\\

Now, suppose $F[0,1)(X)$ is open in $X$. Then $C$ is closed in $X$, and hence compact, making $\con_{x_0}([0,1],C)$ compact as well. As $f$ is continuous, so is its restriction to $[0,1]\ti C$. By what we have already shown, this restriction has image $X$ and has fibres coinciding exactly with the elements (as equivalence classes) of $\con_{x_0}([0,1],C)$. This means $f':\con_{x_0}([0,1],C)\ri X$ is a continuous bijection. As $\con_{x_0}([0,1],C)$ is compact and $X$ is Hausdorff, it follows that $f'$ is a homeomorphism.\\ \end{proof}

\begin{ex} The condition that $F[0,1)(X)$ be open in $X$ is necessary. Consider the family of functions $\{F_{\al}\}_{\al\in [0,1]}$ on $[0,1]^{\om}$ defined so that $F_{\al}((x_i)_{i\in\ma N})=(\al x_i)_{i\in\ma N}$ and equip $[0,1]^{\om}$ with the product topology. $[0,1]^{\om}$ is compact and the family $\{F_{\al}\}_{\al\in [0,1]}$ is easily seen to be a dilation family in $[0,1]^{\om}$ about $(0)_{i\in \ma N}$. However, the set $C=[0,1]^{\om}-\cup_{\al\in [0,1)}F_{\al}([0,1]^{\om})$ consists of exactly those elements $(a_i)_{i\in\ma N}\in [0,1]^{\om}$ such that $\sup\{a_i|i\in \ma N\}=1$. $C$ in this case, is hence not only not a closed subset of $[0,1]^{\om}$, it is a proper dense subset of the ambient space, $[0,1]^{\om}$.\\ \end{ex}

We can apply Proposition \ref{prop: compact space cone condition} to a special case for compact metric spaces:\\

\begin{co}
\label{co: compact metrizable space cone condition} Let $(X,d)$ be a compact metric space. Suppose for some $x_0$ in $X$, there exists a continuous monoidal monomomorphism $F: ([0,1],\ti) \ri (C_{x_0}(X),\circ)$ such that $F(0)=O_{x_0}$ and $F(1)=\tr{Id}_X$. If $d(F(\al)(x),x_0)< d(x,x_0), \fa \al\in [0,1), x\in X-\{x_0\}$ and $x_0$ is not a limit point of $C=X-F[0,1)(X)$ , then for any sufficiently small $\eps\in\ma R^+$, the function $f':\con_{x_0}([0,1],\tr{Bd}(B_D(x_0;\eps)))\ri \overline{B_D(x_0;\eps)}$ defined so that $x_0\mapsto x_0$ and $(\al,c)\mapsto F(\al)(c)$ is a homeomorphism. \end{co}

\begin{proof} As $x_0$ is not a limit point of $C$, we can choose some $\de>0$ such that the open ball $B_D(x_0;\de)$ is disjoint from $C$. Because $d(F(\al)(x),x_0)<d(x,x_0), \fa \al\in [0,1)$, $F[0,1)(\overline{B_D(x_0;\eps)})\subset B_D(x_0;\eps)$. This means $\overline{B_D(x_0;\eps)}-F[0,1)(\overline{B_D(x_0;\eps)})\supset \overline{B_D(x_0;\eps)}-B_D(x_0;\eps)=\tr{Bd}(B_D(x_0;\eps)))$.\\

By the reasoning explained in the proof of Proposition \ref{prop: compact space cone condition}, for each $x\in B_D(x_0;\eps)$, there exists some $c\in C$ and some $\al\in [0,1)$ such that $x=F(\al)(c)$. Now define $\Ga:[\al,1]\ri \ma R$ as $\Ga(\de)=d(F(\de)(c),x_0), \fa \de\in [\al,1]$. This is a continuous real-valued function on a connected interval. We have $\Ga(1)=d(F(1)(c),x_0)=d(c,x_0)\geq\eps$ (this, by the choice of $\eps$) and $\Ga(\al)=d(F(\al)(c),x_0)=d(x,x_0)<\eps$, so there exists some $\be\in (\al,1]$ such that $\Ga(\be)=\eps$, i.e. $d(F(\be)(c),x_0)=\eps$. Setting $y=F(\be)(x)$, we have $y\in \overline{B_D(x_0;\eps)}$. As $\be>\al$, we have that 

$$x=F(\al)(c)=F\left(\frac{\al}{\be}\right)(F(\be)(c))$$

$$=F\left(\frac{\al}{\be}\right)(y)$$

$$\in F[0,1)(\overline{B_D(x_0;\eps)})$$

This holds $\fa x\in X$, hence it follows directly that $B_D(x_0;\eps)\subset F[0,1)(\overline{B_D(x_0;\eps)})\Ri \overline{B_D(x_0;\eps)}-F[0,1)(\overline{B_D(x_0;\eps)})\subset \tr{Bd}(B_D(x_0;1)))$, so $\overline{B_D(x_0;\eps)}-F[0,1)(\overline{B_D(x_0;\eps)})= \tr{Bd}(B_D(x_0;\eps)))$. Since $F(1)=\tr{Id}_X$, $F[0,1](\overline{B_D(x_0;\eps)})\subset \overline{B_D(x_0;\eps)}$, so the result now follows directly from Proposition \ref{prop: compact space cone condition}.\\ \end{proof}

Now, we look at a similar condition as in the statement of Proposition \ref{prop: compact space cone condition} but modified for locally compact spaces with dilation families that both 'contract' and 'expand':\\

\begin{prop}
\label{prop: dilation family existence - locally compact case} Let $(X,\tau)$ be a locally compact Hausdorff space. Suppose for some $x_0$ in $X$ there exists a continuous monoidal monomorphism, $F: ([0,\nf),\ti) \ri (C_{x_0}(X),\circ)$, such that $F(0)=O_{x_0}$ and $F(1)=\tr{Id}_X$. If there is a compact set, $D$, such that $x_0\in \tr{Int}(D)$ and $F[0,1)(D)\subset \tr{Int}(D)$, then $F[0,1)(D)=\tr{Int}(D)$ and the function $f':\con_{x_0}([0,\nf),\tr{Bd}(D))\ri X$ defined so that $x_0\mapsto x_0$ and $(\al,c)\mapsto F(\al)(c)$ is a homeomorphism.\end{prop}

\begin{proof} First, we show that $X$ cannot be compact. Suppose by way of contradiction that $X$ is compact. As $F$ is a monomorphism and $F(0)=O_{x_0}, F(1)=\tr{Id}_X$, we must have that $X\neq\{x_0\}$, i.e. we can choose some point $y\in X$ so that $X-\{y\}$ is an open neighborhood of $x_0$. We also have $F(0)(X)=O_{x_0}(X)=\{x_0\}$, so $F(0)$ belongs to the subbasis element $S(X,X-\{y\})$. As $F$ is continuous, this means we can choose some $\al>0$, such that $F(\be)\in S(X,X-\{y\}), \fa \be\in [0,\al)$. In particular: 

$$F\left(\frac{\al}{2}\right)\in S(X,X-\{y\})$$

$$\Ri F\left(\frac{\al}{2}\right)(X)\subset X-\{y\}\subsetneq X$$

However, we also have that $F\left(\frac{\al}{2}\right)\circ F\left(\frac{2}{\al}\right)=F(1)=\tr{Id}_X$, so we must have $F\left(\frac{\al}{2}\right)(X)=X$ as well, a contradiction. So $X$ cannot be compact.\\

Now, let any point of $X-\{x_0\}$, $x$, be given. I claim that $F[0,\nf)(x)\not\subset D$. Suppose by way of contradiction that $F[0,\nf)(x)\subset D$. This means that $\overline{F[0,\nf)(x)}$ is compact. Also, given any $\al\in [0,\nf)$, we have (by continuity) that 

$$F(\al)(\overline{F[0,\nf)(x)})\subset \overline{F(\al)(F[0,\nf)(x))}$$

$$=\overline{(F(\al)\circ F[0,\nf))(x)}\subset \overline{F[0,\nf)(x)}$$

so each $F(\al)$ maps $\overline{F[0,\nf)(x)}$ into itself. Lastly, if $F(\al)(x)=F(\be)(x)$ for any $\al\neq \be\in [0,\nf)$, then (supposing WLOG that $\al<\be$) $F\left(\frac{\al}{\be}\right)(x)=x$, so $F\left(\left(\frac{\al}{\be}\right)^n\right)(x)=x, \fa n\in \ma Z^+$, but, as $x\neq x_0$, this contradicts the fact that $\lim_{\ga\ri 0} F(\ga)(x)=x_0$, so $F(\al)(x)\neq F(\be)(x), \fa \al\neq \be\in [0,\nf)$. This means the restricted map, $F': ([0,\nf),\ti) \ri (C_{x_0}(\overline{F[0,\nf)(x)}),\circ)$ is also a continuous monoidal monomorphism. However, this contradicts the result proven in the first paragraph. So we must have $F[0,\nf)(x)\not\subset D$.\\

Next, I claim that $F[0,\nf)(x)\cap \tr{Bd}(D)\neq \emptyset$. For otherwise, $X-D$ and $\tr{Int}(D)$ are a pair of disjoint open sets both of which intersect $F[0,\nf)(x)$ (by the last paragraph), contradicting the connectedness of $F[0,\nf)(x)$. So $F(\al)(x)\in \tr{Bd}(D)$ for some $\al\in (0,\nf)$, i.e. $F(\be)(c)=x$ for some $c\in \tr{Bd}(D)$ with $\be=\frac{1}{\al}$. I claim that $F(\al)(c_1)=F(\be)(c_2)$ for any $c_1,c_2\in \tr{Bd}(D)$ only if $\al=\be$ and $c_1=c_2$ or simply just $\al=\be=0$. As if one of $\al$ and $\be$ are $0$, then so must the other be (as neither of $c_1$ and $c_2$ can equal $x_0$). And if both $\al$ and $\be$ are positive, then (assuming WLOG that $\al<\be$) we have:

$$c_2=F(1)(c_2)=F\left(\frac{1}{\be}\right)(F(\be)(c_2))$$

$$=F\left(\frac{1}{\be}\right)(F(\al)(c_1))=F\left(\frac{\al}{\be}\right)(c_1)$$

but as $c_1\in \tr{Bd}(D)\subset D$ and $\frac{\al}{\be}\in (0,1)$, this means $c_2\in F[0,1)(D)\subset \tr{Int}(D)$, contradicting the assumption that $c_2\in \tr{Bd}(D)$.\\

I also claim that if $x\in \tr{Int}(D)-\{x_0\}$, then $\be$ must lie in $(0,1)$ for otherwise, $\frac{1}{\be}\in (0,1)$ (if $\be=1$, then $x=F(1)(c)=c$) and

$$c=F(1)(c)=F\left(\frac{1}{\be}\right)(F(\be)(c))=F\left(\frac{1}{\be}\right)(x)$$

$$\in F\left(\frac{1}{\be}\right)(\tr{Int}(D))\subset F\left(\frac{1}{\be}\right)(D) \subset \tr{Int}(D)$$

a contradiction. It also follows from this that $\tr{Int}(D)\subset F[0,1)(\tr{Bd}(D))\subset F[0,1)(D)$, meaning $\tr{Int}(D)=F[0,1)(D)$.\\

This holds $\fa x\in F[0,1)(D)-\{x_0\}$ and $x_0=F(0)(c)$ for any $c\in \tr{Bd}(D)$, hence it follows that $F[0,1)(D)\subset F[0,1)(\tr{Bd}(D))\Ri F[0,1)(D)=F[0,1)(\tr{Bd}(D))$. Since $F(1)=\tr{Id}_X$, $F[0,1](D)\subset D$. Additionally, as $\tr{Bd}(D)\neq \emptyset$ and $d(F(\al)(x),x_0)<d(x,x_0), \fa \al\in [0,1), \fa x\in U-\{x_0\}$, the restricted map, $F': ([0,\nf),\ti) \ri (C_{x_0}(D),\circ)$ is a continuous monoidal monomorphism as well. So by Proposition \ref{prop: compact space cone condition}, the restricted map, $f'_{[0,1]}:\con_{x_0}([0,1],\tr{Bd}(D))\ri D$, is a homeomorphism.\\

We now need to show that the map $f':\con_{x_0}([0,\nf),\tr{Bd}(D))\ri X$ is a homeomorphism. As $X$ is locally compact and Hausdorff, the restricted evaluation map, $f: [0,\nf)\ti \tr{Bd}(D)\ri X$ is continuous. By what we have already shown, it is also surjective, and its fibres coincide with the members of $\con_{x_0}([0,\nf),\tr{Bd}(D))$, so $f'$ is a continuous bijection. Let $O$ be any open basis element in $\con_{x_0}([0,\nf),\tr{Bd}(D))$. There exists some $\al>0$ such that $O\subset \con_{x_0}([0,\al),\tr{Bd}(D))$. Let $G_{\al}$ denote the map $:\con_{x_0}([0,\nf),\tr{Bd}(D))\ri \con_{x_0}([0,\nf),\tr{Bd}(D))$ that fixes $x_0$ and sends any point $(a,c)$ to $\left(\frac{a}{\al},c\right)$. This map is clearly a homeomorphism. It is now trivial to check that $f'(O)=(F(\al)\circ f'_{[0,1]}\circ G_{\al})(O)$. $G_{\al}$ is a homeomorphism. As $O$ is an open subset of $\con_{x_0}([0,\al),\tr{Bd}(D))$, $G_{\al}(O)$ is an open subset of $\con_{x_0}([0,1),\tr{Bd}(D))$. As $f_{[0,1]}'$ is a homeomorphism, $f_{[0,1]}'(G_{\al}(O))$ is an open subset of $F[0,1)(D)$, which is open in $X$. Finally, as $F(\al)$ is a homeomorphism, we have that $f'(O)$ is open in $X$, making $f'$ a homeomorphism.\\ \end{proof}

There is a similar convenient Corollary to Proposition \ref{prop: dilation family existence - locally compact case} in the case where the locally compact space is also a metric space:

\begin{co}
\label{co: dilation family existence - locally compact metric case} Let $(X,d)$ be a locally compact metric space. Suppose for some $x_0$ in $X$ there exists a continuous monoidal monomorphism, $F: ([0,\nf),\ti) \ri (C_{x_0}(X),\circ)$, such that $F(0)=O_{x_0}$ and $F(1)=\tr{Id}_X$. If $d(F(\al)(x),x_0)< d(x,x_0), \fa \al\in [0,1), x\in X-\{x_0\}$, then for all sufficiently small $\eps>0$, the function $f':\con_{x_0}([0,\nf),\tr{Bd}(B(x_0;\eps))\ri X$ defined so that $x_0\mapsto x_0$ and $(\al,c)\mapsto F(\al)(c)$ is a homeomorphism.\end{co}

\begin{proof} Choose any $\eps$ small enough that $B(x_0;\eps)$ has compact closure. Then, by the shrinking condition, $F[0,1)(\overline{B(x_0;\eps)})\subset B(x_0;\eps)\subset \tr{Int}(\overline{B(x_0;\eps)})$. The result now follows directly from Proposition \ref{prop: dilation family existence - locally compact case}.\\ \end{proof} 

We observe that, if we use the result in Proposition \ref{prop: metrizable cones} in combination with either Corollary \ref{co: compact metrizable space cone condition} or Corollary \ref{co: dilation family existence - locally compact metric case}, we get resulting dilation families on spaces satisfying the Corollary statement conditions that are also linear dilation families. So it turns out that in many situations, the existence of a dilation family on a space can imply a 'flatness' about the center of dilation.\\

We now look at some additional characterizations of dilation families and similar structures that do not involve the same reliance on converting spaces into metrizable cones.\\

\begin{prop}
\label{prop: general space condition} Let $(X,d)$ be a metric space. Let $I$ be a pure set containing $0$ and $F:(I,\ti)\ri (X^X,\circ)$ be a monoidal homomorphism with $F(0)=O_{x_0}$. If for each $x\in X$, there exists some $\al_x, A_x>0$ such that $\frac{\be d(x,y)}{A_x}\leq d(F(\be)(x),F(\be)(y))\leq A_x\be d(x,y), \fa y\in X, \fa \be\in (0,\al_x)$ and $\lim_{\substack{\al\ri 1\\ \al\in I}} F(\al)(x)$ exists, then the function $D:X\ti X\ri \ma R_{\geq 0}$ defined as $(x,y)\ri \lim_{\al\ri 0^+}\sup\left\{\frac{d(F(\be)(x),F(\be)(y))}{\be}\Big|\be\in I\cap (0,\al]\right\}$ defines a metric on $X$ that induces the same topology as $d$ and under which the collection $\{F(\al)\}_{\al\in I}$ forms a dilation family on $(X,d,x_0)$.\end{prop}

\begin{proof} First, we show that $D$ is a metric on $X$. 

Given the inequality condition in the proposition statement, $D$ is clearly a well-defined nonnegative real-valued function.\\

If $x=y$, then 

$$\frac{d(F(\be)(x),F(\be)(y))}{\be}=\frac{d(F(\be)(x),F(\be)(x))}{\be}=\frac{0}{\be}=0$$

so 

$$D(x,y)=\lim_{\al\ri 0^+} \sup\{0\}=\lim_{\al\ri 0^+} 0=0$$

Conversely, if $D(x,y)=0$, there exists a sequence of positive real numbers, $\{a_n\}_{n\in\ma Z^+}$ such that $\lim_{n\ri \nf} a_n=0$ and 

$$\sup\left\{\frac{d(F(\be)(x),F(\be)(y))}{\be}\Big|\be\in I\cap (0,a_n]\right\}< \frac{1}{n} \fa n\in \ma Z^+$$

This means 

$$\frac{d(F(a_n)(x),F(a_n)(y))}{a_n}<\frac{1}{n}, \fa n\in \ma Z^+\Ri\lim_{n\ri \nf}\frac{d(F(a_n)(x),F(a_n)(y))}{a_n}=0$$

As $\lim_{n\ri \nf} a_n=0$, there exists some $M$ such that $\fa n\geq M, a_n<\al_x$. Since 

$$\frac{d(x,y)}{A_x}\leq \frac{d(F(\be)(x),F(\be)(y))}{\be}, \fa \be<\al_x$$

we have:

$$\frac{d(x,y)}{A_x}\leq \lim_{\substack{n\ri \nf\\ n\geq M}} \frac{d(F(a_n)(x),F(a_n)(y))}{a_n}$$

$$=\lim_{n\ri \nf}\frac{d(F(a_n)(x),F(a_n)(y))}{a_n}=0$$

It follows that $d(x,y)=0\Ri x=y$, so $D$ satisfies the identity of indiscernibles.\\

$D$ is also trivially seen to be symmetric.\\

And given any $x,y,z\in X$ and any $\be\in \ma R^+$, we have 

$$\frac{d(F(\be)(x),F(\be)(z))}{\be}\leq \frac{d(F(\be)(x),F(\be)(y))}{\be}+\frac{d(F(\be)(y),F(\be)(z))}{\be}$$

So for any $\al\in \ma R^+$, we have:

$$\sup\left\{\frac{d(F(\be)(x),F(\be)(z))}{\be}\Big|\be\in I\cap (0,\al]\right\}$$

$$\leq \sup\left\{\frac{d(F(\be)(x),F(\be)(y))}{\be}+\frac{d(F(\be)(y),F(\be)(z))}{\be}\Big|\be\in I\cap (0,\al]\right\}$$

$$\leq \sup\left\{\frac{d(F(\be)(x),F(\be)(y))}{\be}\Big|\be\leq \al\right\}+\sup\left\{\frac{d(F(\be)(y),F(\be)(z))}{\be}\Big|\be\in I\cap (0,\al]\right\}$$

Taking limits as $\al\ri 0$ on both sides gives us $D(x,z)\leq D(x,y)+D(y,z)$, making $D$ a metric on $X$.\\

Now we show that $D$ induces the same topology on $X$ as $d$. Let any $x\in X$ and any $\eps>0$ be given. If $y$ is any point in $B_d\left(x;\frac{\eps}{2A_x}\right)$ then for any $\be<\al_x$, we have:

$$\frac{d(F(\be)(x),F(\be)(y))}{\be}\leq A_xd(x,y)<\frac{\eps}{2}$$

So:

$$D(x,y)=\lim_{\al\ri 0^+}\sup\left\{\frac{d(F(\be)(x),F(\be)(y))}{\be}\Big|\be\in I\cap (0,\al]\right\}$$

$$\leq \sup\left\{\frac{d(F(\be)(x),F(\be)(y))}{\be}\Big|\be\in I\cap (0,\al_x]\right\}\leq \frac{\eps}{2}<\eps$$

meaning $B_d\left(x;\frac{\eps}{2A_x}\right)\subset B_D(x;\eps)$.\\

Similarly, let $y$ be any point in $B_D\left(x;\frac{\eps}{A_x}\right)$. We have $D(x,y)<\frac{\eps}{A_x}\Ri \lim_{\al\ri 0^+}\sup\left\{\frac{d(F(\be)(x),F(\be)(y))}{\be}\Big|\be\in I\cap (0,\al]\right\}<\frac{\eps}{A_x}$. This means we can choose some $\ga>0$ such that $\sup\left\{\frac{d(F(\be)(x),F(\be)(y))}{\be}\Big|\be\leq \ga\right\}<\frac{\eps}{A_x}$, so if $\de$ is any positive real that is less than both $\ga$ and $\al_x$, we have:

$$\frac{d(x,y)}{A_x}\leq \frac{d(F(\de)(x),F(\de)(y))}{\de}<\frac{\eps}{A_x}$$

$$\Ri d(x,y)<\eps$$

so $y\in B_d(x;\eps)$, meaning $B_D\left(x;\frac{\eps}{A_x}\right)\subset B_d(x;\eps)$. So $d$ and $D$ both induce the same topology on $X$.\\ 

Lastly, given any $\ga$ in $I$ and any $x,y\in I$, we have:

$$D(F(\ga)(x),F(\ga)(y))=\lim_{\al\ri 0^+}\sup\left\{\frac{d(F(\be)(F(\ga)(x)),F(\be)(F(\ga)(y)))}{\be}\Big|\be\in I\cap (0,\al]\right\}$$

$$=\ga\lim_{\al\ri 0^+}\sup\left\{\frac{d(F(\be\ga)(x),F(\be\ga)(y))}{\be\ga}\Big|\be\in I\cap (0,\al]\right\}$$

$$=\ga\lim_{\al\ri 0^+}\sup\left\{\frac{d(F(\de)(x)),F(\de)(y))}{\de}\Big|\de\in I\cap (0,\al]\right\}$$

$$=\ga D(x,y)$$

where $\de$ in the second-last expression is a substitution for $\be\ga$ (this works because $I$ is a pure set). So $F(\al)$ is a dilation of scale $\al$ under $D$ for each $\al\in I$, completing the proof.\\ \end{proof}

\begin{prop}
\label{prop: translation-invariant metric} Let $(X,*)$ be a topological group with identity $x_0$. Let $d$ be a metric on $X$ that induces said topology. Suppose that there is a dilation family, $\{T_{\al}\}_{\al\in [0,\nf)}$, on $(X,d,x_0)$ consisting of group homomorphisms on $X$ with respect to $*$.  If the collection of translation maps $\{\rho_c\}_{c\in X}$ (where each $\rho_c$ is the map $x\mapsto c*x$) is equicontinuous with respect to $d$, then the function $D:X\ti X\ri \ma R_{\geq 0}$ defined as $D(x,y)=\sup\{d(c*x,c*y)| \fa c\in X\}$ is a metric on $X$ that induces the same topology on $X$ as $d$ does, and $\{T_{\al}\}_{\al\in [0,\nf)}$ is a dilation family on $(X,D,x_0)$. Furthermore, the norm on $X$ defined as $||x||=D(x,x_0), \fa x\in X$ is a homogeneous norm on $X$.\end{prop}

\begin{proof} 

It is easy to see that, as long as we can show that $\{d(c*x,c*y)|\fa c\in X\}$ is bounded for each $(x,y)\in X\ti X$, then $D$ is well-defined.\\ 

Let any $x$ in $X$ be given. As the collection $\{\rho_x\}_{x\in X}$ is equicontinuous with respect to $d$, there exists some open neighborhood, $U$ of $x_0$, such that $\{d(c*x_0,c*z)|\fa c\in X\}$ is bounded for each $z\in U$. Now by Proposition \ref{prop: x-function is cont}, there exists some $\ga>0$ such that $T_{\al}(x)\in U$, so $\{d(c*x_0,c*T_{\ga}(x))|\fa c\in X\}$ is bounded. But this set can be reexpressed as:

$$=\{d(c*T_{\ga}(x_0),c*T_{\ga}(x))|\fa c\in X\}$$

$$=\{d(T_{\ga}(T_{\ga}^{-1}(c)*x_0),T_{\ga}(T_{\ga}^{-1}(c)*x)|\fa c\in X\}$$

As $\ga>0$, $T_{\ga}$ is a bijection on $X$, so (with the substitution $e=T_{\ga}^{-1}(c)$), this set is just:

$$=\{d(T_{\ga}(e*x_0),T_{\ga}(e*x)|\fa e\in X\}$$

$$=\{\ga d(e*x_0,e*x)| \fa e\in X\}$$

So the set $\{\ga d(e*x_0,e*x)| \fa e\in X\}$ must be bounded. As $\ga>0$, this implies that $\{d(e*x_0,e*x)| \fa e\in X\}$ must be bounded, $\fa x\in X$. Now, let any $x,y\in X$ be given. By what we have just shown, $\{d(e*x_0,e*(x^{-1}*y))|\fa e\in X\}$ is bounded. Multiplication by any fixed element of $X$ is a bijection, so, letting $c$ denote $e*x^{-1}$, we have that 

$$\{d(e*x_0,e*(x^{-1}*y))|\fa e\in X\}$$

$$=\{d((c*x)*x_0,(c*x)*(x^{-1}*y)| \fa c\in X\}$$

$$=\{d(c*x,c*y)| \fa c\in X\}$$

and hence this final set is also bounded, $\fa (x,y)\in X\ti X$, meaning $D$ is well-defined.\\

That $D$ is symmetric follows from the fact that $d$ is symmetric.\\

If $x=y$, then $D(x,y)=\sup\{d(c*x,c*y)| \fa c\in X\}=\sup\{d(c*x,c*x)| \fa c\in X\}=\sup\{0\}=0$. Similarly, if $D(x,y)=0$, then $0\leq d(x,y)=d(x_0*x,x_0*y)\leq \sup\{d(c*x,c*y)| \fa c\in X\}=0$, so $d(x,y)=0\Ri x=y$, meaning $D$ satisfies the identity of indiscernables as well.\\

And for any $x,y,z\in X$, we have $d(c*x,c*z)\leq d(c*x,c*y)+d(c*y,c*z), \fa c\in X$, so:

$$D(x,z)=\sup\{d(c*x,c*y)| \fa c\in X\}\leq \sup\{d(c*x,c*y)+d(c*y,c*z)| \fa c\in X\}$$

$$\leq \sup\{d(e*x,e*y)| \fa e\in X\}+\sup\{d(f*y,f*z)| \fa f\in X\}=D(x,y)+D(x,z)$$

meaning $D$ satisfies the triangle inequality as well. Hence, $D$ is a metric on $X$.\\

To show that $D$ induces the same topology on $X$ as $d$, let $x$ be any point in $X$ and $\eps$ be any positive real. As $D(x,y)<\eps\Ri d(x,y)=d(x_0*x,x_0*y)\leq D(x,y)<\eps$, we have $x\in B_D(x;\eps)\subset B_d(x;\eps)$, meaning $D$ induces a topology finer than that induced by $d$. At the same time, as the collection $\{\rho_x\}_{x\in X}$ is equicontinuous with respect to $d$, there exists some $\de>0$ such that $d(c*x,c*y)<\frac{\eps}{2}, \fa c\in X$ and for all $y\in B_d(x;\de)$, meaning $D(x,y)=\sup\{d(c*x,c*y)| \fa c\in X\}\leq \frac{\eps}{2}<\eps$. So $x\in B_d(x;\de)\subset B_D(x;\eps)$, meaning $d$ induces a topology finer than that induced by $D$. So $D$ and $d$ induce equivalent topologies on $X$.\\

Next, for any $x,y\in X$ and any $\al\in (0,\nf)$, we have:

$$D(T_{\al}(x),T_{\al}(y))=\{d(c*T_{\al}(x),c*T_{\al}(y))|\fa c\in X\}$$

$$=\sup\{d(T_{\al}(T_{\al}^{-1}(c)*x),T_{\al}(T_{\al}^{-1}(c)*y)|\fa c\in X\}$$

$$=\sup\{d(T_{\al}(e*x),T_{\al}(e*y)|\fa e\in X\}$$

$$=\sup\{\al d(e*x,e*y)|\fa e\in X\}$$

$$=\al D(x,y)$$

where in the $3$rd equality, $e$ represents $T_{\al}^{-1}(c)$. It follows directly that $\{T_{\al}\}_{\al\in [0,\nf)}$ is a dilation family on $(X,D,x_0)$.\\

Finally, we show that $D$ induces a group norm on $X$. Given any $x,y\in X$ and any $z\in X$, we have:

$$D(z*x,z*y)=\sup\{d(c*(z*x),c*(z*y))| \fa c\in X\}$$

$$=\sup\{d((c*z)*x,(c*z)*y)| \fa c\in X\}$$

$$=\sup\{d(e*x,e*y)|\fa e\in X\}$$

$$=D(x,y)$$

so $D$ is (left) translation-invariant. For any $x\in X$, we now have:

$$||x||=D(x,x_0)=D(x_0,x)$$

$$=D(x^{-1}*x_0,x^{-1}*x)=D(x^{-1},x_0)$$

$$=||x^{-1}||$$

and for any $x,y\in X$, we have:

$$||x*y||=D(x*y,x_0)=D(x^{-1}*(x*y),x^{-1}*x_0)$$

$$=D(y,x^{-1})\leq D(y,x_0)+D(x_0,x^{-1})$$

$$=D(y,x_0)+D(x^{-1},x_0)=||y||+||x^{-1}||$$

$$=||x||+||y||$$

making the induced norm a group norm as well.\\ \end{proof} 

\section{Links with Marius Buliga's work}

First, some results here may be phrased as telling us when metric spaces are 'metric cones'. As defined by M. Buliga (see \cite{intrinsic}, page $10$, Def $2.17$), a metric cone is basically a pointed locally compact metric space $(X,d,x_0)$ with an open ball about $x_0$ that admits dilations into itself about $x_0$ of all scales in $(0,1]$ \cite{intrinsic}. This is somewhat weaker than the condition of a space possessing a dilation family.\\

So we have:

\begin{prop}
\label{prop: metric cone existence - locally compact metric case} Let $(X,d)$ be a locally compact metric space. Suppose for some $x_0$ in $X$ there exists a continuous action $F:[0,\nf)\ti X\ri X$ with $F(0,x)=x_0, \fa x\in X$ ($F(1,x)=x, \fa x\in X$, as well). If $d(F(\al,x),x_0)< d(x,x_0), \fa \al\in [0,1), x\in X-\{x_0\}$, then there exists a metric, $D$ on $X$, such that $D$ and $d$ induce the same topologies on $X$ and $(X,D,x_0)$ is a metric cone.\end{prop}

There is similarly the following result relating to 'normed local groups with dilations' (from \cite{intrinsic}, page $13$, Def $3.2$):

\begin{prop}
\label{prop: local group dilations - locally compact metric case} Let $(X,d)$ be a locally compact metric space such that $X$ is a local group with identity $x_0$. Suppose there exists a continuous action $F:[0,\nf)\ti X\ri X$ with $F(0,x)=x_0, \fa x\in X$ ($F(1,x)=x, \fa x\in X$, as well). If $d(F(\al,x),x_0)< d(x,x_0), \fa \al\in [0,1), x\in X-\{x_0\}$, then there exists a metric, $D$ on $X$, such that $D$ and $d$ induce the same topologies on $X$ and $(X,||\cd ||, \de)$ is a normed local group with dilations with a homogeneous norm where $||x||$ is defined as $D(x,x_0), \fa x\in X$ and $\de_{\eps}(x)=F(\eps,x), \fa (\eps,x)\in [0,\nf)\ti X$.\end{prop}

There is also the following result for 'normed conical groups' (these are normed groups equipped with dilation families such that the dilations are group homomorphisms - see \cite{intrinsic}, page $14$, Def $3.3$):

\begin{prop}
\label{prop: normed group dilations - metric case} Let $(X,d)$ be a metric space such that $X$ is a group with the operation $*$ and identity $x_0$. Suppose also that there is a family of dilations about $x_0$ in $(X,d)$, $\{T_{\al}\}_{\al\in [0,\nf)}$ where each  $T_{\al}$ is a group homomorphism. If the family of translations $\{\rho_c\}_{c\in X}$ where $\rho_c(x)=c*x, \fa c,x\in X$ is equicontinuous at $x_0$, then the function $D:X\ti X\ri \ma R$ defined as $$D(x,y)=\sup\{d(c*x,c*y)|c\in X\}$$ is a translation-invariant metric on $X$ that induces the same topology on $X$ as $d$ and the norm induced by $D$ makes $X$ a normed conical group.\end{prop}

\bibliographystyle{plain}
\bibliography{bibrefs.bib}

\begin{thebibliography}{1}

\bibitem{dilatations}
Marius Buliga.
\newblock Dilatation structures i. fundamentals.
\newblock {\em J. Gen. Lie Theory Appl., Vol1(2007), No.2, 65-9}, 2006.

\bibitem{intrinsic}
Marius Buliga.
\newblock Sub-riemannian geometry from intrinsic viewpoint, 2012.

\end{thebibliography}

\end{document}